\documentclass[a4paper,12pt]{article}
\usepackage[english]{babel}
\usepackage{tikz}
\usetikzlibrary{matrix,arrows,decorations.markings}
\usepackage{amsmath,amsfonts,amssymb,amsthm,url,textcomp}
\usepackage{csquotes}
\usepackage{a4wide}
\usepackage[numbers]{natbib}
\usepackage{fancyhdr}
\usepackage{anyfontsize}
\usepackage{hyperref}
\usepackage{hhline}
\usepackage{leftidx}

\numberwithin{theorem}{section}

\newtheorem{theoremm}{Theorem}\numberwithin{theoremm}{subsection}
\newtheorem{deffinition}[theoremm]{Definition}
\newtheorem{lemmma}[theoremm]{Lemma}
\newtheorem{corrollary}[theoremm]{Corollary}
\newtheorem{propposition}[theoremm]{Proposition}
\numberwithin{theoremmm}{subsubsection}

\theoremstyle{remark}

\newcommand{\Rad}{\operatorname{Rad}}
\newcommand{\Aut}{\operatorname{Aut}}

\newcommand{\PSL}{\operatorname{PSL}}

\newcommand{\Aff}{\operatorname{Aff}}

\newcommand{\lcm}{\operatorname{lcm}}
\newcommand{\sh}{\operatorname{sh}}
\newcommand{\ord}{\operatorname{ord}}

\newcommand{\Hol}{\operatorname{Hol}}
\newcommand{\A}{\operatorname{A}}
\newcommand{\LL}{\operatorname{L}}

\newcommand{\mao}{\operatorname{mao}}
\newcommand{\maffo}{\operatorname{maffo}}

\newcommand{\Out}{\operatorname{Out}}

\newcommand{\GL}{\operatorname{GL}}

\newcommand{\rel}{\mathrm{rel}}

\newcommand{\cha}{\operatorname{char}}

\begin{document}

\title{A bound on element orders in the holomorph of a finite group}

\author{Alexander Bors\thanks{University of Salzburg, Mathematics Department, Hellbrunner Stra{\ss}e 34, 5020 Salzburg, Austria. \newline E-mail: \href{mailto:alexander.bors@sbg.ac.at}{alexander.bors@sbg.ac.at} \newline The author is supported by the Austrian Science Fund (FWF):
Project F5504-N26, which is a part of the Special Research Program \enquote{Quasi-Monte Carlo Methods: Theory and Applications}. \newline 2010 \emph{Mathematics Subject Classification}: 20C25, 20D05, 20D45. \newline \emph{Key words and phrases:} Finite groups, Holomorph (group theory), Element orders, Bound.}}

\date{\today}

\maketitle

\abstract{Let $G$ be a finite group. We prove a theorem implying that the orders of elements of the holomorph $\Hol(G)$ are bounded from above by $|G|$, and we discuss an application to bounding automorphism orders of finite groups.}

\section{Introduction}\label{sec1}

\subsection{Motivation and main results}\label{subsec1P1}

Holomorphs are frequently encountered in permutation group theory. For example, it is well-known that a permutation group $G$ acting on a set $X$ and having a regular normal subgroup $N$ is the internal semidirect product of $N$ and the point stabilizer $G_x$ for any $x\in X$, and since the conjugation action of $G_x$ on $N$ is faithful, one obtains a natural embedding $G\hookrightarrow\Hol(N)$. Three of the eight O'Nan-Scott types of finite primitive permutation groups (HA, HS and HC) are of this form.

Conversely, for any group $G$, $\Hol(G)$ admits a natural faithful permutation representation on (the underlying set of) $G$ in which the canonical copy of $G$ in $\Hol(G)$ is regular; this is by letting $\Hol(G)$ act on $G$ via what the author called \textit{affine maps} in \cite[Definition 2.1.1]{Bor15a}. In this action, the element $(x,\alpha)\in\Hol(G)$ corresponds to the permutation $\A_{x,\alpha}:G\rightarrow G$ sending $g\mapsto x\alpha(g)$. We denote the image of this permutation representation (i.e., the group of bijective affine maps of $G$) by $\Aff(G)$.

Our motivation for studying holomorphs of finite groups lies in the search for upper bounds on automorphism orders. By \cite[Lemma 2.1.4]{Bor15a}, we have the following: If $G$ is a group, $\alpha$ an automorphism of $G$, $x\in G$, $H$ an $\alpha$-invariant subgroup of $G$ and $gH$ a coset of $H$ such that $\A_{x,\alpha}[gH]\subseteq gH$, say $\A_{x,\alpha}(g)=gh_0$ with $h_0\in H$, then the action of $\A_{x,\alpha}$ on $gH$ is isomorphic (in the sense of an isomorphism of finite dynamical systems, see \cite[remarks after Definition 1.1.5]{Bor15a}) with the action of the bijective affine map $\A_{h_0,\alpha_{\mid H}}$ on $H$. Using this, we can prove:

\begin{propposition}\label{lcmProp}
Let $G$ be a finite group, $N \cha G$ and $A=\A_{x,\alpha}\in\Aff(G)$. Denote by $\tilde{A}$ the induced affine map on $G/N$. Then there exists a subset $M\subseteq N$ such that \[\ord(A)=\ord(\tilde{A})\cdot\lcm_{m\in M}{\ord(\A_{m,(\alpha_{\mid N})^{\ord(\tilde{A})}})}.\]
\end{propposition}

\begin{proof}
Clearly, $\ord(\tilde{A})\mid\ord(A)$, and $A^{\ord(\tilde{A})}$ restricts to a permutation on each coset of $N$ in $G$. By the remarks before this proposition, for each coset $C\in G/N$, we can fix an element $n_C\in N$ such that the action of $A^{\ord(\tilde{A})}$ on $C$ is isomorphic with the action of $\A_{n_C,\alpha_{\mid N}^{\ord(\tilde{A})}}$ on $N$. Set $M:=\{n_C\mid C\in G/N\}$. Then clearly, $\ord(A^{\ord(\tilde{A})})=\lcm_{m\in M}{\ord(\A_{m,\alpha_{\mid N}^{\ord(\tilde{A})}})}$, and the result follows.
\end{proof}

This led the author to studying the following function on finite groups:

\begin{deffinition}\label{fDef}
For a finite group $G$, define \[\mathfrak{F}(G):=\max_{\alpha\in\Aut(G)}\lcm_{x\in G}{\ord(\A_{x,\alpha})}.\]
\end{deffinition}

Clearly, $\mathfrak{F}(G)$ is an upper bound on the maximum element order in $\Hol(G)$, and thus both on the maximum element and maximum automorphism order of $G$. Our main result is the following upper bound on $\mathfrak{F}(G)$:

\begin{theoremm}\label{fTheo}
For any finite group $G$, we have $\mathfrak{F}(G)\leq|G|$. In particular, element orders in $\Hol(G)$ are bounded from above by $|G|$.
\end{theoremm}

It is not difficult to show that $\mathfrak{F}(G)=|G|$ whenever $G$ is a finite cyclic or dihedral group, whence this upper bound is in general best possible.  We remark that it is known \cite[Theorem 2]{Hor74a} that the maximum automorphism order of a nontrivial finite group $G$ is bounded from above by $|G|-1$. Furthermore, we note that our proof of Theorem \ref{fTheo} will use the classification of finite simple groups (CFSG). Before tackling the proof, we note an easy consequence. For a finite group $G$, denote by $\mao(G)$ the maximum automorphism order of $G$, by $\maffo(G)$ the maximum order of a bijective affine map of $G$ (which coincides with the maximum element order in $\Hol(G)$), and set $\mao_{\rel}(G):=\mao(G)/|G|$ and $\maffo_{\rel}(G):=\maffo(G)/|G|$.

\begin{corrollary}\label{fCor}
For any finite group $G$ and any characteristic subgroup $N$ of $G$, we have:

(1) $\mao_{\rel}(G/N)\geq\mao_{\rel}(G)$ ($\mao_{\rel}$ is increasing on characteristic quotients).

(2) $\maffo_{\rel}(G/N)\geq\maffo_{\rel}(G)$ ($\maffo_{\rel}$ is increasing on characteristic quotients).
\end{corrollary}

\begin{proof}
For (1), fix an automorphism $\alpha$ of $G$ such that $\ord(\alpha)=\mao(G)$. In view of Proposition \ref{lcmProp} and Theorem \ref{fTheo}, we deduce that $\mao(G)=\ord(\alpha)\leq\ord(\tilde{\alpha})\cdot\mathfrak{F}(N)\leq\mao(G/N)\cdot |N|$, and the result follows upon dividing both sides of the inequality by $|G|$. The proof of (2) is analogous.
\end{proof}

Corollary \ref{fCor} extends \cite[Lemma 4.3]{Bor15c}, which dealt with the special case $N=\Rad(G)$, the solvable radical of $G$.

\section{On the proof of Theorem \ref{fTheo}}\label{sec2}

\subsection{Some auxiliary results}\label{subsec2P1}

In this subsection, we present some results used in the proof of Theorem \ref{fTheo}. We begin by restating \cite[Lemma 2.1.6]{Bor15b} for the readers' convenience:

\begin{lemmma}\label{divisorLem}
Let $G$ be a finite group, $x\in G$, $\alpha$ an automorphism of $G$. Then every cycle length of $\A_{x,\alpha}$ is divisible by $\LL_G(x,\alpha):=\ord(\sh_{\alpha}(x))\cdot\prod_{p}{p^{\nu_p(\ord(\alpha))}}$, where $p$ runs through the common prime divisors of $\ord(\sh_{\alpha}(x))$ and $\ord(\alpha)$. In particular, $\LL_G(x,\alpha)\mid|G|$.\qed
\end{lemmma}

We can use this to give some sufficient conditions for least common multiples as in the definition of $\mathfrak{F}(G)$ to be bounded by $|G|$:

\begin{lemmma}\label{lcmDivLem}
Let $G$ be a finite group, $\alpha\in\Aut(G)$.

(1) If $\ord(\alpha)\mid |G|$, then $\lcm_{x\in G}{\ord(\A_{x,\alpha})}\mid |G|$.

(2) For every prime $p\mid |G|$, we have \[\lcm_{x\in G}{\ord(\A_{x,\alpha})}\mid\prod_{q\mid|G|,q\not=p}{q^{\nu_q(|G|)}}\cdot p^{2\nu_p(\exp(G))}\cdot\exp(\Out(G)).\] In particular, if, for some prime $p\mid |G|$, we have \[p^{2\nu_p(\exp(G))}\cdot\exp(\Out(G))\leq p^{\nu_p(|G|)},\] then $\lcm_{x\in G}{\ord(A_{x,\alpha})}\leq |G|$.
\end{lemmma}

\begin{proof}
For (1): Fix $x\in G$. We will show that $\ord(\A_{x,\alpha})$. which equals $\ord(\alpha)\cdot\ord(\sh_{\alpha}(x))$, divides $|G|$. This is tantamount to proving that for any prime $p$, we have $\nu_p(\ord(\alpha))+\nu_p(\ord(\sh_{\alpha}(x)))\leq\nu_p(|G|)$. This is clear (\textit{inter alia} by assumption) if $p$ divides at most one of the two numbers $\ord(\alpha)$ and $\ord(\sh_{\alpha}(x))$, and if $p$ divides both these numbers, the inequality holds by Lemma \ref{divisorLem}.

For (2): Again, we fix $x\in G$. We shall prove that \[\ord(\alpha)\cdot\ord(\sh_{\alpha}(x))\mid\prod_{q\mid|G|,q\not=p}{q^{\nu_q(|G|)}}\cdot p^{2\nu_p(\exp(G))}\cdot\exp(\Out(G)).\] Denoting by $\pi:\Aut(G)\rightarrow\Out(G)$ the canonical projection and noting that $\ord(\alpha)=\ord(\pi(\alpha))\cdot\ord(\alpha^{\ord(\pi(\alpha))})$ with $\ord(\pi(\alpha))\mid\exp(\Out(G))$, we find that it is sufficient to prove that $\ord(\alpha^{\ord(\pi(\alpha))})\cdot\ord(\sh_{\alpha}(x))\mid\prod_{q\mid|G|,q\not=p}{q^{\nu_q(|G|)}}\cdot p^{2\nu_p(\exp(G))}$. Fix a prime $l$. If $l$ divides at most one of the numbers $\ord(\alpha^{\ord(\pi(\alpha))})$ and $\ord(\sh_{\alpha}(x))$, it is clear that the corresponding inequality of $l$-adic valuations holds. Hence assume that $l$ divides both these numbers. If $l\not=p$, we are done by an application of Lemma \ref{divisorLem}, and if $l=p$, we are done since both orders divide $p^{\nu_p(\exp(G))}$.
\end{proof}

In view of Lemma \ref{lcmDivLem}, the following well-known technique for bounding the $p$-exponent of a finite group, particularly of a finite group of Lie type with defining characteristic $p$, will be useful:

\begin{lemmma}\label{lieTypeLem}
Let $p$ be a prime, $K$ a field of characteristic $p$, $d\in\mathbb{N}^+$. Let $A\in\GL_d(K)$ be of finite order. Then $\nu_p(\ord(A))\leq\lceil\log_p(d)\rceil$. In particular, denoting by $d_p(G)$ the minimum faithful projective representation degree in characteristic $p$ of the finite group $G$, we have $\nu_p(\exp(G))\leq\lceil\log_p(d_p(G))\rceil$.\qed
\end{lemmma}

Finally, we note that the function $\mathfrak{F}$ satisfies an inequality which is useful for proofs by induction:

\begin{lemmma}\label{csubLem}
For all finite groups $G$ and $N \cha G$, we have $\mathfrak{F}(G)\leq\mathfrak{F}(N)\cdot\mathfrak{F}(G/N)$.
\end{lemmma}

\begin{proof}
Fix an automorphism $\alpha$ of $G$ such that $\mathfrak{F}(G)=\lcm_{x\in G}{\ord(\A_{x,\alpha})}=:L$. Denote by $\tilde{\alpha}$ the automorphism of $G/N$ induced by $\alpha$, by $\pi:G\rightarrow G/N$ the canonical projection, and set $L_1:=\lcm_{y\in G/N}{\ord(\A_{y,\tilde{\alpha}})}$. Clearly, $L_1\leq\mathfrak{F}(G/N)$. On the other hand, setting $L_2:=\lcm_{x\in G}{\ord(\A_{x,\alpha}^{L_1})}$, since each $\ord(\A_{x,\alpha})$ divides $L_1\cdot L_2$, $L$ divides and thus is bounded from above by $L_1\cdot L_2$, so it suffices to show that $L_2\leq\mathfrak{F}(N)$. Now as in the proof of Proposition \ref{lcmProp}, each $\ord(\A_{x,\alpha}^{L_1})$ is a least common multiple of orders of bijective affine maps of $N$ of the form $\A_{n,(\alpha_{\mid n})^{L_1}}$ for various $n\in N$. But then $L_2$ itself is also a least common multiple of such orders, and thus bounded from above by $\mathfrak{F}(N)$, as we wanted to show.
\end{proof}

\subsection{Proof of Theorem \ref{fTheo}}\label{subsec2P2}

The proof is by induction on $|G|$, with the induction base $|G|=1$ being trivial. For the induction step, note that if $G$ is not characteristically simple, then fixing any proper nontrivial characteristic subgroup $N$ of $G$, we have, by Lemma \ref{csubLem} and the induction hypothesis, $\mathfrak{F}(G)\leq\mathfrak{F}(N)\cdot\mathfrak{F}(G/N)\leq |N|\cdot |G/N|=|G|$. Hence we may assume that $G$ is characteristically simple, i.e., $G=S^n$ for some finite (not necessarily nonabelian) simple group $S$ and $n\in\mathbb{N}^+$.

The case where $S$ is abelian, i.e., $S=\mathbb{Z}/p\mathbb{Z}$ for some prime $p$, is treated by \cite[Lemma 4.3]{Bor15c}, so we may assume that $S$ is nonabelian. Let us first treat the case $n\geq 2$. Note that by \cite[Lemma 3.4]{Bor15b} and \cite[Theorem 1]{Hor74a}, we have $\mao(S^n)<|S^n|^{0.438}$. Furthermore, $\exp(S^n)=\exp(S)\leq |S|\leq |S^n|^{0.5}$. It follows that $\lcm_{x\in S^n}{\ord(\A_{x,\alpha})}=\ord(\alpha)\cdot\lcm_{x\in S^n}{\ord(\sh_{\alpha}(x))}\leq |S^n|^{0.438}\cdot |S^n|^{0.5}<|S^n|$.

We may thus henceforth assume that $G=S$ is a nonabelian finite simple group. It is well-known that the Sylow $2$-subgroups of $S$ are not cyclic, whence we are done by Lemma \ref{lcmDivLem}(1) if $\exp(\Out(S))\leq 2$. This settles all alternating and all sporadic $S$.

Now assume that $S$ is of Lie type. We will treat this case mostly by applications of Lemma \ref{lcmDivLem}(2), with $p$ always equal to the defining characteristic of $S$. Hence our goal is to show the inequality $p^{2\nu_p(\exp(S))}\cdot\exp(\Out(S))\leq p^{\nu_p(|S|)}$, which we do by means of Lemma \ref{lieTypeLem}. Information on $|S|$ and $|\Out(S)|$ is available from \cite[p. xvi, Tables 5 and 6]{CCNPW85a}, and the values of $d_p(S)$ for the various finite simple groups of Lie type can be found in \cite[p. 200, Table 5.4.C]{KL90a}.

It is straightforward to verify the sufficient inequality $p^{2\lceil\log_p(d_p(S))\rceil}\cdot|\Out(S)|\leq p^{\nu_p(|S|)}$ for $S=\PSL_2(p^f)$ with $f\geq 3$, with the exception of the cases $(p,f)=(2,3),(3,3),(5,3)$, for $S=\PSL_d(q)$ with $d\geq 3$, with the exception of $(d,q)=(3,2),(3,4)$, and for all $S$ of Lie type which are not isomorphic with any $\PSL_d(q)$.

For $S=\PSL_2(p)$ with $p\geq 5$ or $S=\PSL_2(p^2)$ with $p\geq 3$, we note that $\exp(\Out(S))=2$, whence we are done as in the alternating and sporadic case. The same applies to $S=\PSL_3(2)$. Finally, one can check with GAP \cite{GAP4} that for $S=\PSL_2(8),\PSL_2(27),\PSL_2(125),\PSL_3(4)$, all automorphism orders of $S$ divide $|S|$, whence Lemma \ref{lcmDivLem}(1) can be applied to conclude the proof.\qed

\end{document}